\documentclass[11pt,a4paper]{article}
\usepackage[OT1]{fontenc}
\usepackage[english]{babel}
\usepackage{amsfonts}
\usepackage{amsthm}
\def\h{ {\cal H} }
\def\a{ {\cal A} }

\def\l{ {\cal L} }
\def\n{ {\cal N} }

\def\b{ {\cal B} }
\def\u{ {\cal U} }
\def\m{ {\cal M} }

\def\p{ {\cal P} }

\def\k{ {\cal K} }
\def\f{ {\cal F} }
\def\c{ {\cal C} }

\newtheorem{teo}{Theorem}[section]
\newtheorem{prop}[teo]{Proposition}
\newtheorem{lem}[teo]{Lemma}
\newtheorem{coro}[teo]{Corollary}

\theoremstyle{definition}
\newtheorem{rem}[teo]{Remark}
\newtheorem{ejem}[teo]{Example}

\title{Uncertainty principle and geometry of the infinite Grassmann manifold}
\author{Esteban Andruchow and Gustavo Corach}

\begin{document}

\maketitle 

\begin{abstract}
We study the pairs of projections
$$
P_If=\chi_If  ,\  \ Q_Jf= \left(\chi_J \hat{f}\right)\check{\ } , \ \ f\in L^2(\mathbb{R}^n),
$$ 
where $I, J\subset \mathbb{R}^n$ are sets of finite Lebesgue measure, $\chi_I, \chi_J$ denote the corresponding characteristic functions and $\hat{\ } , \check{\ }$ denote the Fourier-Plancherel transformation $L^2(\mathbb{R}^n)\to L^2(\mathbb{R}^n)$ and its inverse. These pairs of projections have been widely   studied by several authors  in connection with the mathematical formulation of Heisenberg's uncertainty principle. Our study is done from a differential geometric point of view.  We apply known results on the Finsler geometry of the Grassmann manifold $\p(\h)$ of a Hilbert space $\h$ to establish that  there exists a unique minimal  geodesic of $\p(\h)$, which is a curve of the form
$$
\delta(t)=e^{itX_{I,J}}P_Ie^{-itX_{I,J}}
$$
which joins $P_I$ and $Q_J$ and has length $\pi/2$. As  a consequence we obtain that if $H$ is the logarithm of the Fourier-Plancherel map, then
$$
\|[H,P_I]\|\ge \pi/2.
$$
The spectrum of $X_{I,J}$  is denumerable and symmetric with respect to the origin, it has a smallest  positive eigenvalue $\gamma(X_{I,J})$ which satisfies
$$
\cos(\gamma(X_{I,J}))=\|P_IQ_J\|.
$$
\end{abstract}

\bigskip



\section{Introduction}
Consider the following example:

\begin{ejem}\label{el ejemplo}
Let $I,J\subset \mathbb{R}^n$ be Lebesgue-measurable sets of finite measure. Let $P_I,Q_J$ be the projections in $L^2(\mathbb{R}^n,dx)$ given by
$$
P_If=\chi_If \ \ \hbox{ and } \ \ Q_Jf= \left(\chi_J \hat{f}\right)\check{\ },
$$
where $\chi_L$ denotes the characteristic function of the set $L$.
Equivalently, denoting by $U_\f$ the Fourier transformation regarded as a unitary operator acting in $L^2(\mathbb{R}^n,dx)$ and by $M_\varphi$ the multiplication by $\varphi$, then
$$
P_I=M_{\chi_I} \ \hbox{ and } Q_J=U_\f^*P_JU_\f .
$$
The operator  $P_IQ_J$ is Hilbert-Schmidt (see for instance \cite{donoho}, Lemma 2).
\end{ejem}

An intuitive formulation of Heisenberg's uncertainty principle says that a nonzero function and its Fourier transform
cannot be (simultaneously) sharply localized (see \cite{folland}, page 207). We give more precision to this statement below  ( see for instance \cite{donoho}, page 906).

According to Folland and Sitaram \cite{folland}, the idea of using projections $P_I$ and $Q_J$ to obtain a form of the uncertainty principle is due to Fuchs \cite{fuchs}, and it was developed later in a series of papers by Landau, Pollack and Slepian \cite{lp1}, \cite{lp2}, \cite{sp}. See the survey by Folland and Sitaram \cite{folland}.

Donoho and Stark \cite{donoho}  proved that if  $I,J\subset \mathbb{R}^n$
with finite Lebesgue measure and $f\in L^2(\mathbb{R}^n)$ with $\|f\|_2=1$ 
 satisfy that
$$
\int_{\mathbb{R}^n-I}|f(t)|^2 d t < \epsilon_I \ \hbox{ and } \ 
\int_{\mathbb{R}^n-J}|\hat{f}(w)|^2 d w < \epsilon_J
$$ 
then 
$$
|I||J|\ge (1-(\epsilon_I+\epsilon_J))^2.
$$
Donoho and Stark showed several applications of these ideas to signal processing (and the obstruction to the 
existence of an instantaneous frequency).  Smith \cite{smith} generalized these results to a locally compact abelian group $G$ where $I\subset G$ and $J\subset \hat{G}$, the dual group of $G$. The books by Havin and 
J\"oricke \cite{havin},  Hogan and Lakey \cite{hogan}, and Gr\"ochenig \cite{grochenig}  among many others, contain further applications, generalizations and history of the different uncertainty principles.

By an elementary computation using Fubini's theorem, Donoho and Stark  prove that 
$$
\|P_IQ_J\|_{HS}=\sqrt{|I||J|},
$$
where $\|\ \|_{HS}$ is Hilbert-Schmidt norm.
 Next they prove  that 
 $$
 \|P_IQ_J\|\ge 1-\epsilon_I-\epsilon_J.
 $$
 The fact that $\|P_IQ_J\|\le \|P_IQ_J\|_{HS}$ is well known. 

They argue that any bound $c$ such that
 $$
  \|P_IQ_J\|\le c<1
  $$
  is an expression of the uncertainty principle (\cite{donoho}, page 912).
  
Denote by $\p(\h)$ the set of orthogonal projections of the Hilbert space $\h$, also called the Grassmann manifold of $\h$. It is indeed a differentiable manifold of $\b(\h)$ (also in the infinite dimensional setting), with rich geometric structure (see for instance \cite{pr} or \cite{cpr}).  The pairs $(P_I,Q_J)$ might be put in the broader  context   of the sets
  $$
  \c=\{(P,Q): P,Q\hbox{ are orthogonal projections and } PQ \hbox{ is compact}\}.
  $$
This set is a $C^\infty$-submanifold of $\p(\h)\times \p(\h)$.

An application of these geometrical results facts is a form of the uncertainty principle (see Theorem \ref{desigualdad conmutante} below).

Let us describe the content of the paper.

In Section 2 we recall the known facts on the geometry of $\p(\h)$. In section 3 we apply known results \cite{pr}, \cite{cpr}, \cite{pemenoscu} on the Finsler geometry of the Grassmann manifold of $\h$ to the special case of pairs $P_I, Q_J$. We prove that there exists a unique minimal geodesic of the Grassmann manifold of length $\pi/2$ which joins $P_I$ and $Q_J$.  That is, there exists a unique selfadjoint operator $X_{I,J}$ of norm $\pi/2$, which is co-diagonal with respect both to $P_I$ and $Q_J$, such that
$$
e^{iX_{I,J}}P_Ie^{-iX_{I,J}}=Q_J.
$$
The spectrum of the operator $X_{I,J}$ is denumerable and symmetric with respect to the origin. The smallest positive eigenvalue $\gamma(X_{I,J})$ verifies
$$
\cos(\gamma(X_{I,J}))=\|P_IQ_J\|.
$$

As a consequence from the fact that the minimal geodesic has length $\pi/2$, we prove that if $H$ is the logarithm of the Fourier transform in $L^2(\mathbb{R}^n)$, and $I\subset \mathbb{R}^n$ is a set of finite Lebesgue measure, then
$$
\|[H,P_I]\|=\|[H,Q_I]\|\ge \pi/2.
$$
In Section 4 we show  that for any pair of sets $I,J\subset\mathbb{R}^n$ of finite measure, 
one has 
$$
N(P_I)+N(Q_J)=L^2(\mathbb{R}^n),
$$ where the sum is non-direct (the subspaces have infinite dimensional intersection).

\section{Basic properties}
\subsection{Halmos decomposition}

Let $\h$ be a Hilbert space, $\b(\h)$ the algebra of bounded linear operators in $\h$, $\k(\h)$ the ideal of compact operators and $\p(\h)$ the set of selfadjoint (orthogonal) projections, and $\p_\infty(\h)$ the subset of projections whose nullspaces and ranges have infinite dimension. 

A tool that will be useful in the study of the pairs $P_I,Q_J$ is  {\it Halmos decomposition } \cite{halmos}, which is the following orthogonal decomposition of $\h$: given a pair of projections $P$ and $Q$, consider
$$ 
\h_{11}=R(P)\cap R(Q) \ , \ \ \h_{00}=N(P)\cap N(Q) \ , \ \h_{10}=R(P)\cap N(Q) \ , \ \ \h_{01}=N(P)\cap R(Q)
$$
and $\h_0$ the orthogonal complement of the sum of the above. This last subspace is usually called the {\it generic part} of the pair $P, Q$. Note also that
$$
N(P-Q)=\h_{11}\oplus\h_{00}\ , \ \ N(P-Q-1)=\h_{10} \ \hbox{ and } N(P-Q+1)=\h_{01},
$$
so that the generic part depends in fact of the difference $P-Q$. 

Halmos proved that there is an isometric isomorphism between $\h_0$ and a product Hilbert space $\l\times\l$ such that in the above decomposition (putting $\l\times\l$ in place of $\h_0$), the projections are 
$$
P=1\oplus 0 \oplus 1 \oplus 0 \oplus \left(\begin{array}{cc} 1 & 0 \\ 0 & 0 \end{array} \right)
$$
and 
$$
Q=1\oplus 0 \oplus 0 \oplus 1 \oplus \left(\begin{array}{cc} C^2 & CS \\ CS & S^2 \end{array} \right),
$$
where $C=cos(X)$ and $S=sin(X)$ for some operator $0<X\le \pi/2$ in $\l$ with trivial nullspace.

Aparently, the pair $(P,Q)$ belongs to $\c$ if and only if  $\h_{11}$ is finite dimensional and $C=cos(X)$ is compact.
\begin{rem}
If $(P,Q)\in\c$, then the spectral resolution of $X$ can be easily described. Since $0<cos(X)$ is compact, it follows that
$$
X=\sum_n \gamma_n P_n+\frac{\pi}{2}E,
$$
where $0<\gamma_n< \pi/2$ is an increasing (finite or infinite) sequence. For all $n$, $\dim R(P_n)<\infty$, and
$$
R(E)\oplus(\oplus_{n\ge 1} R(P_n))=\l.
$$

\end{rem}

\subsection{Finsler geometry of the Grassmann manifold of $\h$}
Let us recall some basic facts on the differential geometry of the set $\p(\h)$ (see for instance \cite{cpr}, \cite{pr}, \cite{pemenoscu}).

\begin{enumerate}
\item
The space $\p(\h)$ is a homogeneous space under the action of the unitary group $\u(\h)$ by inner conjugation: if $U\in\u(\h)$ and $P\in\p(\h)$, the action is given by
$$
U\cdot P=UPU^*.
$$
This action is locally transitive: it is well known that two projections $P_1, P_2$ such that $\|P_1-P_2\|<1$, are conjugate.  Therefore, since the unitary group $\u(\h)$ is connected,  the orbits of the action coincide with the connected components of $\p(\h)$, which are: for $n\in\mathbb{N}$,  $\p_{n,\infty}(\h)$ (projections of nullity $n$), $\p_{\infty,n}(\h)$ (projections of rank $n$) and $\p_{\infty}(\h)$ (projections of infinite rank and nullity). These components are $C^\infty$-submanifolds of $\b(\h)$.
\item
There is a natural linear connection in $\p(\h)$. If $\dim \h <\infty$, it is the Levi-Civita connection of the Riemannian metric which consists of considering the Frobenius inner product at every tangent space. It is based on the diagonal / co-diagonal decomposition of $\b(\h)$. To be more specific, given $P_0\in\p(\h)$, the tangent space of $\p(\h)$ at $P_0$ consists of all selfadjoint co-diagonal matrices (in terms of  $P_0$). The linear connection in $\p(\h)$ is induced by a reductive structure, where the horizontal elements at $P_0$ (in the Lie algebra of $\u(\h)$: the space of antihermitian elements of $\b(\h)$) are the co-diagonal antihermitian operators. The geodesics of $\p$ which start at $P_0$ are curves of the form
\begin{equation}\label{geodesica}
\delta(t)=e^{itX}P_0e^{-itX},
\end{equation}
with $X^*=X$ co-diagonal with respect to $P_0$.  Observe that $X$ is co-diagonal with respect to every $P_t=\delta(t)$. 
It was proved in \cite{pr} that if $P_0,P_1\in\p(\h)$ satisfy $\|P_0-P_1\|<1$, then there exists a unique geodesic (up to reparametrization) joining $P_0$ and $P_1$. This condition is not necessary for the existence of a unique geodesic. 
\item 
There exists a unique geodesic joining two projections $P$ and $Q$ if and only if
$$
R(P)\cap N(Q)=N(P)\cap R(Q)=\{0\},
$$
(see \cite{pemenoscu}).
\item
If $\h$ is infinite dimensional, the Frobenius metric is not available. However, if one endows each tangent space of $\p(\h)$ with the usual norm of $\b(\h)$, one obtains a continuous (non regular) Finsler metric,
$$
d(P_0,P_1)=\inf\{ \ell(\gamma): \gamma \hbox{  a continuous  piecewise smooth curve in } \p(\h) \hbox{ joining } P_0 \hbox{ and } P_1\}
$$
where $\ell(\gamma)$ denotes the length of $\gamma$ (parametized in the interval $I$):
$$
\ell(\gamma)= \int_I \|\dot{\gamma}(t)\| d t .
$$
 In \cite{pr} it was shown that the geodesics (\ref{geodesica}) remain minimal among their endpoints for all $t$ such that 
$$
|t|\le \frac{\pi}{2\|X\|}.
$$
It can be shown  that $d(P_0,P_1)<\pi/2$ if and only if $\|P_0-P_1\|<1$. In other words,  $\|P_0-P_1\|=1$ if and only if $d(P_0,P_1)=\pi/2$.
\end{enumerate}

\section{Geometry of the pairs $P_I$, $Q_J$}
Lenard proved in \cite{lenard} that  the projections
$P_I,Q_J\in\p(L^2(\mathbb{R}^n,dx))$ defined in Example (\ref{el ejemplo}), satisfy 
\begin{equation}\label{nulo}
R(P_I)\cap N(Q_J)=R(Q_J)\cap N(P_I)=\{0\}.
\end{equation}
Moreover, $\|P_I-Q_J\|=1$.

Therefore one obtains the following:
\begin{teo}
Let $I, J$ be measurable subsets of $\mathbb{R}^n$ of finite measure, and $P_I$, $Q_J$ the above projections.
Then there exists a unique selfadjoint operator $X_{I,J}$ satisfying:
\begin{enumerate}
\item
$\|X_{I,J}\|=\pi/2$.
\item
$X_{I,J}$ is $P_I$ and $Q_J$ co-diagonal. In other words, $X_{I,J}$ maps functions in $L^2(\mathbb{R}^n,dx)$ with support in $I$ to functions with support in $\mathbb{R}^n- I$, and functions such that $\hat{f}$ has support in $J$ to functions such that the Fourier transform has support in $\mathbb{R}^n- J$.
\item
$e^{iX_{I,J}}P_Ie^{-iX_{I,J}}=Q_J$.
\item
If $P(t)$, $t\in[0,1]$ is a smooth curve in $\p(\h)$ with $P(0)=P_I$ and $P(1)=Q_J$, then
$$
\ell(P)=\int_0^1\|\dot{P}(t)\| d t \ge \pi/2.
$$
\end{enumerate}
\end{teo}
\begin{proof}
By the condition (\ref{nulo}) above (\cite{lenard}), it follows from \cite{pemenoscu} that there exists a unique minimal geodesic of $\p(\h)$, of the form
$$
\delta_{I,J}(t)=e^{itX_{I,J}}P_Ie^{itX_{I,J}}
$$
with $X_{I,J}^*=X_{I,J}$ co-doagonal with respect to $P_I$ (and $Q_J$) such that
$$
\delta_{I,J}(1)=Q_J.
$$
Condition 4. above is the minimality property of $\delta_{I,J}$. Finally, the fact that $\|P_I-Q_J\|=1$ means that $\|X_{I,J}\|=\pi/2$. 
\end{proof}

\begin{rem}
 It is known \cite{folland} that $\lambda_1=\|P_IQ_JP_I\|=\|P_IQ_J\|^2<1$, and moreover $\sqrt{\lambda_1}$ equals the cosine of the angle between the subspaces $R(P_I)$ and $R(Q_J)$.
 
 One can also relate this number $\lambda_1$ with the operator $X_{I,J}$. Using Halmos decomposition (recall that it consists only of $\h_{00}$ and the generic part $\h_0$ in this case),
 $$
 P_IQ_JP_I=0\oplus \left( \begin{array}{cc} C^2 & 0 \\ 0 & 0 \end{array} \right)
 $$
 and thus $\lambda_1=\|cos(X)\|^2$. We shall see below that the spectrum of $X$ is a strictly  
 increasing sequence of positive eigenvalues 
 $\gamma_n\to \pi/2$, with finite multiplicity.
 Moreover, since $P_IQ_JP_I$
 belongs to $\b_1(\h)$, it follows that $C\in\b_2(\l)$. Thus 
 $$
 \{cos(\gamma_n)\}\in\ell^2.
 $$
 \end{rem}
For a given $P\in\p(\h)$, let $\a_P$ be 
$$
\a_P=\{X\in\b(\h): [X,P] \hbox{ is compact}\}.
$$
Apparently $\a_P$ is a C$^*$-algebra.
 \begin{teo}
 Let $I, J$ be measurable subsets of $\mathbb{R}^n$ of finite Lebesgue measure.
 \begin{enumerate}
 \item
 The selfadjoint operator $X_{I,J}$ has closed infinite dimensional range, in particular it is not compact. 
 \item
 Let $I_0$ be another measurable set with finite measure such that $|I\cap I_0|=0$, and let $P_0=P_{I_0}$.  Then, the  commutant $[X_{I,J},P_0]$ is compact.
\end{enumerate}
 \end{teo}
 \begin{proof}
 Easy matrix computations (\cite{pemenoscu}) show that, in the  decomposition $\h=\h_{00}\oplus (\l\times\l)$, $X_{I,J}$ is of the form
 $$
 X_{I,J}=0\oplus \left( \begin{array}{cc} 0 & -iX \\ iX & 0 \end{array} \right).
$$
 Note that the spectrum of this operator is symmetric  with respect to the origin. Indeed, if $V$ equals the symmetry 
 $$
 V=1\oplus \left( \begin{array}{cc} 0 & 1 \\ 1 & 0 \end{array} \right),
 $$
 then apparently $VX_{I,J}V=-X_{I,J}$. Also note that
 $$
 X_{I,J}^2=0\oplus \left( \begin{array}{cc} X^2 & 0 \\ 0 & X^2 \end{array} \right).
 $$
 Therefore the spectrum of $X_{I,J}$ is
 $$
 \sigma(X_{I,J})=\{0\}\cup\{ \gamma_n: n\ge 1\}\cup \{ -\gamma_n: n\ge 1\},
 $$
 with $0$ of infinite multiplicity, and the multiplicity of $\gamma_n$ equal to the multiplicity of $-\gamma_n$, and finite.
 What matters here, is that the set $\{\gamma_n:n\ge 1\}$ is infinite, and is therefore an increasing 
 sequence converging to $\pi/2$. This holds because otherwise, 
 the operator $C$ would have finite rank, and therefore $P_IQ_JP_I$ would be of finite rank, which is not the case 
 (see \cite{lenard}). Thus $X_{I,J}$ has closed range. of infinite dimension.
 
  Note that $P_I$ and $Q_J$ satisfy that 
    $P_IP_0=0$ and  $Q_JP_0=Q_JP_{I_0}$ is compact, and therefore $P_I, Q_J\in\a_{P_0}$. Thus the  symmetries $S_{P_I}, S_{Q_J}$ belong to $a_{P_0}$.
Since  $S_{Q_J}=e^{i2X_{I,J}}S_{P_I}$, this implies that
$$
e^{i2X_{I,J}}\in\a_{P_0}.
$$
By the spectral picture of $X_{I,J}$ it is clear that  $X_{I,J}$ can be obtained as an holomorphic function of $e^{i2X_{I,J}}$. 
Since $\a_{P_0}$ is a C$^*$-algebra, this implies that $X_{I,J}\in\a_{P_0}$. 
 \end{proof}

 Let us relate the operator $X_{I,J}$ with the mathematical version of the uncertainty principle, according to \cite{donoho} and \cite{folland}. 
 
 Let $A\in\b(\h)$ be an operator with closed range, the {\it reduced minimum modulus} $\gamma_A$ of $A$ is the positive number
 $$
 \gamma_A=\min\{\|A\xi\|: \xi\in N(A)^\perp, \|\xi\|=1\}=\min\{|\lambda|: \lambda\in\sigma(A), \lambda\ne 0\}.
 $$
 Donoho and Stark \cite{donoho} underline the role of the number $\|Q_JP_I\|$ and consider any constant $c$ such that $\|Q_JP_I\|\le c$ a manifestation of the (mathematical) uncertainty principle. By the above Remark, we have:
 \begin{coro}
  With the current notations, 
  $$
  \|Q_JP_I\|=\cos(\gamma_{X_{I.J}}).
  $$
  \end{coro}
  \begin{proof}
  Indeed, in the above description of the spectrum of $X_{I,J}$, the reduced minimum modulus  $\gamma_{X_{I.J}}$ of $X_{I,J}$ coincides with $\gamma_1$.
  \end{proof} 
 Let $X_{I,J}^0$ be the restriction of $X_{I,J}$ to the generic part of $P_I$ and $Q_J$, i.e.,  its restriction to $N(X_{I,J})^\perp$. In Halmos decomposition
$$
X_{I,J}^0=\left( \begin{array}{cc} 0 & -iX \\ iX & 0 \end{array} \right) .
$$
Recall the formula by Donoho and Stark \cite{donoho}
$$
\|P_IQ_J\|_{HS}=|I|^{1/2}|J|^{1/2}.
$$
From the preceeding facts,  it also follows:
\begin{coro}
With the current notations
$$
|I|^{1/2}|J|^{1/2}=\|\cos(X)\|_{HS}=\frac{1}{\sqrt{2}}\|cos(X_{I,J}^0)\|_{HS}=\{\sum_{n=1}^\infty \frac12\cos(\gamma_n)^2\}^{1/2}.
$$
\end{coro}
\begin{proof}
$$
|I||J|=\|P_IQ_J\|^2_{HS}=Tr(P_IQ_JP_I)=Tr(C^2)=\frac12 Tr\left( \begin{array}{cc} C^2 & 0 \\ 0 & C^2 \end{array}\right)=\frac12 Tr(\cos(X_{I,J}^0)^2).
$$
\end{proof}

This co-diagonal exponent  $X_{I,J}$ (with respect both to $P_I$ and $Q_J$) has  interesting features when $I=J$ and $|I|<\infty$. 
 In this case  denote by $X_I=X_{I,I}$; then,   we have two unitary operators intertwining $P_I$ and $Q_I$. 
 Namely, the Fourier transform $U_{\f}$ and the exponential $e^{iX_I}$,
 $$
 U_\f^*P_IU_\f=Q_I= e^{iX_I}P_Ie^{-iX_I}.
 $$
  Let $H=H^*$ be the natural logarithm of the Fourier transform, $e^{iH}=U_\f$. Namely, writing $E_1$, $E_{-1}$, $E_i$ and $E_{-i}$ the eigenprojections of $U_\f$, 
$$
H=-\pi  E_{-1}+\frac{\pi}{2} E_i- \frac{\pi}{2} E_{-i}.
$$
Note that $\|H\|=\pi$.  Thus,  one obtains a smooth path joining $P_I$ and $Q_I$:
 $$
 \varphi(t)=e^{-itH}P_Ie^{itH}.
 $$
 and, apparently,  $\varphi(1)=Q_I$.
 
 Since the Fourier transform intertwines $P_I$ and $Q_J$, the norm of its commutant with either of these projections can be regarded as a measure of non commutativity between $P_I$ and $Q_J$:
 \begin{teo}\label{desigualdad conmutante}
  For any Lebesgue measurable set $I\subset\mathbb{R}^n$ with $|I|<\infty$, one has
  $$
  \|[H,P_I]\|=\|[H,Q_I]\|\ge \pi/2.
  $$
 \end{teo}
\begin{proof}
 The geodesic $\delta_I$ with exponent $X_I$ is the shortest curve in $\p(\h)$ joining $P_I$ and $Q_I$. Its length is $\pi/2$. Then
 $$
 \pi/2\le \ell(\varphi)=\int_0^1\|\dot{\varphi}(t)\| d t =\int_0^1\|e^{itH}[H,P_I]e^{-itH}\| d t=\|[H,P_I]\|.
 $$
 Note that 
 $$
 U_\f^*[H,P_I]U_\f=[H,U_\f^* P_IU_\f]=[H,Q_I]
 $$
 because $U_\f$ and $H$ commute.
\end{proof}
\begin{rem}

\noindent
\begin{enumerate}
\item
We may write $H$ in terms of $U_\f$ using the well known formulas
$$
E_{-1}=\frac14(1-U_\f+U_\f^2-U_\f^3), \ E_i=\frac14(1-i U_\f-U_\f^2+ i U_\f^3) , \ E_{-i}=\frac14(1+i U_\f- U_\f^2-i U_\f^3),
$$
and thus
$$
H=\frac{\pi}{4}\{-1 +(1+i)U_\f-U_\f^2+(1+i)U_\f^3\}.
$$
Then
$$
[H,P_I]=\frac{\pi}{4}\{(1+i)[U_\f,P_I]-[U_\f^2,P_I]+(1+i)[U_\f^3,P_I]\}.
$$
The inequality in Corollary \ref{desigualdad conmutante} can be written
$$
\|(1+i)[U_\f,P_I]-[U_\f^2,P_I]+(1+i)[U_\f^3,P_I]\|\ge 2.
$$
\item
In the special case when the set $I$ is (essentially) symmetric with respect to the origin, 
$P_I$ commutes with $U_\f^2$, so that 
$$
[U_\f^2,P_I]=0 \ \hbox{ and } \ \ [U_\f^3,P_I]=[U_\f,P_I]U_\f^2=U_\f^2[U_\f,P_I]
$$
one has
$$
[H,P_I]=\frac{(1+i)\pi}{4}[U_\f,P_I](1+U_\f^2).
$$
The operator $U_\f^2f(x)=f(-x)$ is a symmetry, then $\frac12(1+U_\f^2)$ is the orthogonal projection $E_e$ onto the the subspace of essentially even functions ($f(x)=f(-x)$ $a.e.$). Then one can write
$$
[H,P_I]=\frac{(1+i)\pi}{2}[U_\f,P_I]E_e=\frac{(1+i)\pi}{2}E_e[U_\f,P_I].
$$

\end{enumerate}
\end{rem}

\begin{coro}
Suppose that $I$ is essentially symmetric, with finite measure.
\begin{enumerate}
\item
$$
\|E_e[U_\f,P_I]\|=\|E_e[U_\f,P_I]E_e\|\ge \frac{1}{\sqrt{2}}.
$$
\item
$$
\|E_eP_I-E_eQ_I\|\ge\frac{1}{\sqrt{2}},
$$
where $E_eP_I=P_IE_e$ and $E_eQ_I=Q_IE_e$ are orthogonal projections.
\end{enumerate}
\end{coro}
\begin{proof}
Recall that $E_e$ and $U_\f$ commute. Then
$$
E_e[U_\f,P_I]E_e=E_e(U_\f P_I-P_IU_\f)E_e
=U_\f E_e(P_I -U_\f^*P_IU_\f) E_e
$$
$$
=U_\f E_e(P_I-Q_I) E_e.
$$
where $E_e$, as well as $U_\f$, and thus also $Q_I=U_\f^* P_I U_\f$ commute with $E_e$.
\end{proof} 
The ranges of these two orthogonal projections $E_eP_I$ and $E_eQ_I$ consist of the elements of $L^2$ which are essentially even and 
vanish (essentially) outside $I$, and the analogous subspace for the Fourier transform.
\section{Spatial properties of $P_I$ and $Q_J$}

Let us return to the general setting ($I$ not necessarily equal to $J$). The ranges and nullspaces of $P_I$ and $Q_J$ have several  interesting properties. First we need the following lemma:
\begin{lem}
 Let $P, Q$ be orthogonal projections such that $\|P-Q\|=1$. Then one and only one of the following conditions hold: 
\begin{enumerate}
\item
$N(P)+R(Q)=\h$, with non direct sum (and this is equivalent to $R(P)+N(Q)$ being a direct sum and a closed proper subspace of $\h$).
\item 
$R(P)+N(Q)=\h$, with non direct sum (and this is equivalent to $N(P)+R(Q)$ being a direct sum and a closed proper subspace of $\h$). 
\item
$R(P)+N(Q)$ is non closed (and this is equivalent to $N(P)+R(Q)$ being non closed).
\end{enumerate}
 \end{lem}
\begin{proof}
By the Krein-Krasnoselskii-Milman formula (see for instance \cite{krein}) 
$$
\|P-Q\|=\max\{\|P(1-Q)\|, \|Q(1-P)\|\},
$$
we have that one and only one of the following hold:
\begin{enumerate}
 \item $\|P(1-Q)\|<1$ and $\|Q(1-P)\|=1$,
 \item $\|P(1-Q)\|=1$ and $\|Q(1-P)\|<1$, or
 \item $\|P(1-Q)\|=1$ and $\|Q(1-P)\|=1$.
\end{enumerate}
This alternative corresponds precisely with the three conditions in the Lemma. It is known \cite{deutsch} that for two orthogonal projections $E$ and $F$,
$\|EF\|<1$ holds if and only if $R(E)\cap R(F)=\{0\}$ and $R(E)+R(F)$ closed. The sum $\m+\n$ of two subspaces is closed if and only if the sum $\m^\perp + \n^\perp$ is closed (see \cite{deutsch}). 
Therefore, $\|EF\|<1$ is also equivalent to $N(E)+N(F)=\h$. 

If we apply these facts to $E=P$ and $F=1-Q$, we obtain that the first alternative is equivalent to $R(P)\cap N(Q)=\{0\}$ and $R(P)+N(Q)$ closed, or to $N(P)+R(Q)=\h$. 

Analogously, the second alternative is equivalent to $R(Q)\cap N(P)=\{0\}$ and $R(Q)+N(P)$ closed, or to $N(Q)+R(P)=\h$.

Note that in the first case, $R(P)+N(Q)$ is proper, otherwise its orthogonal complement would be $N(P)\cap R(Q)=\{0\}$, which together with the fact that $N(P)+R(Q)=\h$ (closed!), would lead us to the second alternative.

Analogously in the second alternative, $N(P)+R(Q)$ is proper.

If neither of these two happen, it is clear that  neither $R(P)+N(Q)$ nor (equivalently) the sum of the orthogonals $N(P)+R(Q)$ is closed.
\end{proof}
We have the following:
\begin{teo}
Let $I,J\subset\mathbb{R}^n$ with finite Lebesgue measure. Then
\begin{enumerate}
\item
$R(P_I)+R(Q_J)$ is a closed proper subset of $L^2(\mathbb{R}^n)$, with infinite codimension. The sum is direct ($R(P_I)\cap R(Q_J)=\{0\}$).
\item
$N(P_I)+N(Q_J)=L^2(\mathbb{R}^n)$, and the sum is not direct ($N(P_I)\cap N(Q_J)$ is infinite dimensional).
\item
$R(P_I)+N(Q_J)$ and $N(P_I)+R(Q_J)$ are proper dense subspaces of $L^2(\mathbb{R}^n)$, and $R(P_I)\cap N(Q_J)=N(P_I)\cap R(Q_J)=\{0\}$.
\end{enumerate}
\end{teo}
\begin{proof}
By the cited result \cite{deutsch},  two projections $P, Q$,  satisfy that $R(P)+R(Q)$ is closed and $R(P)\cap R(Q)=\{0\}$  if and only if $\|PQ\|<1$. It is also known (see above, \cite{folland}) that $\|P_IQ_J\|<1$. The intersection of these spaces is, in our case (using the notation of the Halmos decomposition)
$$
R(P_I)\cap R(Q_J)=\h_{11}=\{0\}.
$$
As remarked above, Lenard proved that $\h_{11}=\h_{10}=\h_{01}=\{0\}$, and $\h_{00}$ is infinite dimensional. The orthogonal complement of this sum is
$$
(R(P_I)+ R(Q_J))^\perp=N(P_I)\cap N(Q_J)=\h_{00}.
$$
Thus the first assertion follows.

In our case  $\|P_I-Q_J\|=1$ (\cite{folland}, \cite{lenard}) thus we may apply the above Lemma.

The first condition cannot happen: 
$$
(N(P_I)+R(Q_J))^\perp=R(P_I)\cap N(Q_J)=\h_{10}=\{0\}.
$$
By a similar argument, neither the second condition can happen. Thus $R(P_I)+R(Q_J)$ is non closed, and its orthogonal complement is trivial. 
Thus the second and third assertions follow.
\end{proof}

\begin{rem} It is known (see for instance \cite{feshenko}), that if $P,Q$ are  projections with $PQ$ compact and $R(P)\cap R(Q)=\{0\}$, then
 $$\|PQ\|<1.$$
\end{rem}

In \cite{cm}, the second named author and A. Maestripieri studied the set  of operators $T\in\b(\h)$ which are of the form $T=PQ$.
Among other properties, they proved that $T$ may have many factorizations, but there is a minimal factorization 
(called {\it canonical factorization} of $T$), 
namely
$$
T=P_{\overline{R(T)}}P_{N(T)^\perp},
$$
which satisfies that if $T=PQ$, then $R(T)\subset R(P)$ and $N(T)^\perp\subset R(Q)$ (or equivalently $N(Q)\subset N(T)$). Following this notation,
\begin{prop}
 The factorization $P_IQ_J$ is canonical.
\end{prop}
\begin{proof}
 Put $T=P_IQ_J$. Using Halmos decomposition in this particular case ($\h=\h_{00}\oplus (\l\times \l)$), apparently
 $$
 P_IQ_JP_I=0\oplus \left( \begin{array}{cc} C & 0 \\ 0 & 0 \end{array} \right),
 $$
 and thus $R(P_IQ_JP_I)=0\oplus (R(C)\times 0)$. Recall that $C^2>0$, and thus $C^2$ has dense range. It follows that
 $$
 \overline{R(T)}=\overline{R(P_IQ_J)}=\overline{R(P_IQ_JP_I)}=0\oplus (\l\times 0),
 $$
 which is precisely the range of $P_I$: $\overline{R(T)}=R(P_I)$. 
  Note the following elementary  fact:
 $$
 N(PQ)=N(Q)\oplus (R(Q)\cap N(P)).
 $$
 For the factorization $T=P_IQ_J$ it is known (\cite{lenard}) that $R(Q_J)\cap N(P_I)={0}$. Thus 
 $$
 N(T)=N(P_IQ_J)= N(Q_J)
 $$
 and the proof follows.
 \end{proof}
 
 In \cite{cm} it is proven that if $T=PQ=P_0Q_0$, and the latter is the canonical factorization, then
 $$
 \|P_0f-Q_0f\|\le \|Pf-Qf\|
 $$
 for any $f\in L^(\mathbb{R}^n)$. In particular $\|P_0-Q_0\|\le \|P-Q\|$. In our case we get the following result
\begin{coro}
Let $P,Q$ projections in $L^2(\mathbb{R}^n)$ such that $PQ=P_IQ_J$. Then for any $f\in L^2(\mathbb{R}^n)$ one has 
$$
\|P_If-Q_Jf\|_2\le \|Pf-Qf\|_2.
$$
In particular, $\|P_I-Q_J\|\le \|P-Q\|$.
\end{coro}

{\sc (Esteban Andruchow)} {Instituto de Ciencias,  Universidad Nacional de Gral. Sar\-miento,
J.M. Gutierrez 1150,  (1613) Los Polvorines, Argentina and Instituto Argentino de Matem\'atica, `Alberto P. Calder\'on', CONICET, Saavedra 15 3er. piso,
(1083) Buenos Aires, Argentina.}

{\sc (Gustavo Corach)} {Instituto Argentino de Matem\'atica, `Alberto P. Calder\'on', CONICET, Saavedra 15 3er. piso, (1083) Buenos Aires, Argentina, and Depto. de Matem\'atica, Facultad de Ingenier\'\i a, Universidad de Buenos Aires, Argentina.}

\end{document}